\documentclass[12pt]{amsart}
\usepackage{amsmath}
\usepackage{amssymb}
\usepackage{amsfonts,amsthm,amscd}
\numberwithin{equation}{section}

\newtheorem{theorem}{Theorem}[section]
\newtheorem{cor}[theorem]{Corollary}

\newtheorem{prop}[theorem]{Proposition}
\theoremstyle{definition}
\newtheorem{df}[theorem]{Definition}
\newtheorem*{acknowledgments}{Acknowledgments}
\theoremstyle{remark}

\newtheorem{example}{Example}

\begin{document}
 \author{Gabriel B\u adi\c toiu}
\address{G. B\u adi\c toiu, ``Simion Stoilow'' Institute of Mathematics  of
the Romanian Academy, P.O. Box 1-764, 014700 Bucharest, Romania
\tt{Gabriel.Baditoiu@imar.ro}}
 \author{Stere Ianu\c s}
\address{S. Ianu\c s, Faculty of Mathematics, University of Bucharest, Romania}
 \author{Anna Maria Pastore}
\address{A.M. Pastore, Dipartimento di Matematica, Universit\` a di Bari,
Italia \tt{pastore@dm.uniba.it}}
\title[Spectral geometry of Riemannian Legendre foliations]{Spectral geometry of Riemannian Legendre foliations}
\thanks{Stere Ianu\c s passed away on April 8th, 2010.}

\begin{abstract}
We obtain geometric characterizations of isospectral minimal Riemannian Legendre foliations on
compact Sasakian manifolds of constant $\varphi$-sectional curvature.
\end{abstract}
\maketitle

\section{Preliminaries}
Let  $\mathcal F$ be a Riemannian foliation on an $m$-dimensional compact Riemannian manifold $(M,g)$.
  We denote by $L$ and $Q=L^\perp$ the tangent and normal bundles of   $\mathcal F$,
and that gives the decomposition of the tangent bundle $TM= L \oplus
L^{\perp}.$ Let $\Delta_g$ be the  Laplace operator associated to
$g$ and let $\nabla$ be the Bott connection of the normal bundle
$Q=TM\,/L$ (see \cite[pp.~20-21]{ton}). The Jacobi operator
${\mathcal J}_{\nabla}$ of $\mathcal F$, defined by ${\mathcal
J}_{\nabla}s =(d^*_{\nabla}d_{\nabla}-\rho_{\nabla})s$ for any $s$
section of the normal bundle,
 is a second order elliptic operator (see \cite{To}).
The compactness of $M$ implies that the spectra of  $\Delta_g$
and ${\mathcal J}_{\nabla}$ are discrete. Using Gilkey's theory (\cite{Gi,Rosenberg}),
 one can write their associated asymptotic expansions:
$$ Tr\,e^{-t\Delta_g}
  = \sum\limits_{i=1}^{\infty} e^{-t\lambda_i}\quad {}_{\widetilde {t \searrow 0}}
  \quad (4\pi t)^{-\frac{m}{2}}\sum\limits_{s=0}^{\infty} t^s a_s(\Delta_g)$$
$$ Tr\,e^{-t{\mathcal J}_{\nabla}}
  = \sum\limits_{i=1}^{\infty} e^{-t\mu_i}\quad {}_{\widetilde {t \searrow 0}}
  \quad (4\pi t)^{-\frac{m}{2}}\sum\limits_{s=0}^{\infty} t^s b_s({\mathcal J}_{\nabla})$$
where
$$a_s(\Delta_g)=\int_M a_s(x,\Delta_g)dv_g$$
$$b_s({\mathcal J}_{\nabla}) = \int_M b_s (X, {\mathcal J}_{\nabla})dv_g$$
are invariants of $\Delta_g$ and ${\mathcal J}_{\nabla}$  depending  only on their corresponding discrete spectra
$$Spec(M,g)=\{0\leq \lambda_1 \leq \lambda_2 \leq \ldots \leq \lambda_i \leq \ldots \uparrow \infty \}$$
$$Spec({\mathcal F}, {\mathcal J}_{\nabla})=\{\mu_1 \leq \mu_2 \leq \ldots \leq \mu_i \leq \ldots \uparrow \infty \}$$
We restrict our attention to the first coefficients  $a_s$ and $b_s$ for $s \in \{0,1,2\}$
which encodes certain properties of the spectral geometry of
 $(M,\mathcal F)$. We recall the following theorem from  \cite{Gi, NiToVa}.
\begin{theorem}\label{T1}
Let $\mathcal F$ be a Riemannian foliation of  codimension $q \geq 2$ on a compact Riemannian manifold $(M,g)$.
Then
\begin{eqnarray}
\begin{array}{ll}
a_0(\Delta_g)& =  a_0 = Vol_g(M)\cr
\\
a_1(\Delta_g) &= a_1 = \frac{1}{6} \int_M \tau dv_g \cr
\\
a_2(\Delta_g) & = a_2 = \frac{1}{360} \int_M (2 \Vert R \Vert^2 -2 \Vert \rho \Vert^2 + 5 \tau^2)dv_g
\end{array}
\end{eqnarray}
\begin{eqnarray}
\begin{array}{ll}
b_0(\mathcal J _{\nabla})& = b_0 = qVol_g(M)\cr
\\
b_1(\mathcal J _{\nabla}) & =b_1 = q a_1 + \int_M \tau_{\nabla} dv_g \cr
\\
b_2(\mathcal J _{\nabla}) & = b_2
= q a_2 + \frac{1}{12}\int_M (2\tau \tau_{\nabla}+6\Vert \rho_{\nabla}\Vert^2 - \Vert R_{\nabla}\Vert^2)dv_g\,,
\end{array}
\end{eqnarray}
where $R, \rho $ are the Riemann and the Ricci tensor fields, $\tau
$ is the scalar curvature of $g$, and $R_{\nabla}, \rho_{\nabla},
\tau_{\nabla}$ are those associated to the Bott connection
 $\nabla$ of the transverse bundle $Q= TM\,/L$.
\end{theorem}
The following theorem, due to  Nishikawa, Tondeur and Vanhecke \cite{NiToVa},
is fundamental for the spectral geometry of a Riemannian foliation.
\begin{theorem}\label{T2} Let $(M,g)$ and $(M_0,g_0)$ be
two compact  Riemannian manifolds endowed with
 Riemannian foliations  $\mathcal F$ and  ${\mathcal  F}_0$
of codimensions $q$ and $q_0$, respectively. If $\mathcal F$ and  ${\mathcal  F}_0$
are isospectral, that is
$$Spec(M,g) = Spec(M_0,g_0), \qquad Spec(\mathcal F, {\mathcal J}_{\nabla})
=Spec({\mathcal F}_0,{\mathcal J}_{\nabla_0}),$$
then the following hold:
\begin{itemize}
\item[\rm i)] $\dim\,M = \dim\,M_0\,, \quad \mathrm{Vol}(M)=\mathrm{Vol}(M_0)\,, \quad q=q_0$,
\item[\rm ii)] $\int_M \tau dv_g = \int_{M_0} \tau_0 dv_{g_0}\,, \quad \int_M \tau_{\nabla} dv_g
= \int_{M_0} \tau_{\nabla_0} dv_{g_0}$,
\item[\rm iii)] $\int_M (2 \Vert R \Vert^2 -2 \Vert \rho \Vert^2 + 5 \tau^2)dv_g =
\int_{M_0} (2 \Vert R_0 \Vert^2 -2 \Vert \rho_0 \Vert^2 + 5 \tau_0^2)dv_{g_0}$,
\item[\rm iv)] $
\int_M (2\tau \tau_{\nabla}+6\Vert \rho_{\nabla}\Vert^2 - \Vert R_{\nabla}\Vert^2)dv_g =
\int_{M_0} (2\tau_0 \tau_{\nabla_0}+6\Vert \rho_{\nabla_0}\Vert^2 - \Vert R_{\nabla_0}\Vert^2)dv_{g_0}$.
\end{itemize}
\end{theorem}
We recall that in the one-codimension case the isospectral Riemannian foliations are completely
determined by the spectrum of $\Delta_g$ and for this reason we shall assume throughout the paper
that the codimension $q\geq 2$.
\section{Spectral Invariants of a Riemannian Legendre foliation}

In this section we compute the spectral invariants $a_s$, $b_s$, for $s\in\{0,1,2\}$ of a
 Riemannian Legendre foliation with minimal leaves on a Sasakian manifold $M$ of constant
$\varphi$-sectional curvature
and then we obtain certain geometric properties of two such isospectral Riemannian foliations.
 First, we recall the notion of Riemannian Legendre foliation.

\begin{df} Let $M$ be a  $(2n+1)$-dimensional compact manifold endowed with a  Sasakian
structure $(\varphi, \xi, \eta, g)$
 and let  $\mathcal D = Ker\,\eta = Im\,\varphi$ be the $2n$-dimensional
distribution on $M$ orthogonal to
 the $1$-dimensional  distribution  generated by $\xi$.
 A Riemannian foliation $\mathcal L$ on $M$ is said to be a
{\bf Riemannian Legendre foliation} if the leaves are
$n$-dimensional and $L_x \subset {\mathcal D}_x$, for each $x \in
M$. Note that $\varphi(L)\subset L^{\perp} $.
\end{df}

Let $(M,\varphi, \xi,\eta,g)$ be a Sasakian manifold of constant
$\varphi$-sectional curvature $c$ and of dimension $2n+1 \geq 5$. We
recall that $\varphi$ is an endomorphism of tangent bundle, $\xi$ is
a vector field on $M$, $\eta$ is the 1-form dual to $\xi$ with
respect to $g$, satisfying:
$$\varphi^2=-I+\eta\otimes\xi,\ \  \eta(\xi)=1, \ \ \varphi(\xi)=0,
\  \ \ \varphi(\eta(X))=0,$$
$$g(X,Y)=g(\varphi (X), \varphi (Y))+\eta(X)\eta(Y),$$
$$
\nabla^M_X \xi =-\varphi(X), \ \ \ \
(\nabla^M_X\varphi)(Y)=g(X,Y)\xi-\eta(Y)X,$$
 for any vector fields $X$, $Y$. By \cite{Bl},
its curvature tensor is given by
\begin{eqnarray}
\begin{array}{ll}
R(X,Y)Z =& \frac{c+3}{4}\{g(Y,Z)X-g(X,Z)Y\}\cr
\\
& + \frac{c-1}{4}\{g(Z,\varphi Y)\varphi X -g(Z,\varphi X)\varphi Y\cr
\\
& +2g(X, \varphi Y)\varphi Z -g(Y,Z)\eta(X) \xi + g(X,Z)\eta(Y)\xi \cr
\\
& -\eta(Y)\eta(Z)X + \eta(X)\eta(Z)Y\}\,,
\end{array}
\end{eqnarray}
and its Ricci tensor and scalar curvature satisfy
\begin{equation}
\rho(X,Y)= \frac{n(c+3)+c-1}{2} g(X,Y) - \frac{(n+1)(c-1)}{2} \eta(X) \eta(Y)\,,
\end{equation}
\begin{equation}\label{e:tau2.3}
\tau =\frac{n}{2}(2n+1)(c+3)+ \frac{n}{2}(c-1)\,.
\end{equation}
Let $\mathcal L$ be a
 Riemannian Legendre foliation on $M$ and $(e_i, \varphi e_i, \xi)$, $ i \in \{1,\ldots, n\}$
 be a local orthonormal basis of $TM$ adapted to the foliation $\mathcal L$, which means that
 $(e_1 ,\ldots , e_n)$ is a local basis of $L$ and $(\varphi e_1,\ldots \varphi e_n)$ is a basis of  $\varphi (L)$.
The curvature tensor writes as
\begin{eqnarray}\label{curvature}
\begin{array}{ll}
R(e_i,\xi,e_k,\xi)=R(\varphi e_i,\xi, \varphi e_k ,\xi)= \delta_{ik}, \cr
\\
R(e_i,e_j, e_k,e_m)= R(\varphi e_i,\varphi e_j,\varphi e_k,\varphi e_m)
=\frac{c+3}{4}(\delta_{ik}\delta_{jm}-\delta_{im}\delta_{jk}),\cr
\\
R(e_i,e_j,\varphi e_k,\varphi e_m)= R(\varphi e_i,\varphi e_j, e_k,
e_m)=\frac{c-1}{4}(\delta_{ik}\delta_{jm}-\delta_{im}\delta_{jk}),\cr
\\
R(e_i,\varphi e_j,e_k,\varphi e_m)=\frac{c+3}{4}\delta_{ik}\delta_{jm}
+ \frac{c-1}{4}\delta_{im}\delta_{jk} + \frac{c-1}{2}\delta_{ij}\delta_{km},
\end{array}
\end{eqnarray}
the other expressions being equal to zero.
The Ricci tensor is given by
\begin{eqnarray}\label{ric}
\begin{array}{ll}
\rho(e_i,\xi)= \rho(\varphi e_i,\xi)= \rho(e_i,\varphi e_j)=0, \cr
\\
\rho(\xi,\xi)=2n,\cr
\\
\rho(e_i,e_j)= \rho(\varphi e_i,\varphi e_j)=\frac{n(c+3)+ c-1}{2}\delta_{ij}.
\end{array}
\end{eqnarray}
The square of the Hilbert-Schmidt norm of $R$, defined to be
\begin{equation}
\Vert R \Vert^2 = \sum_{a,b,c,d} g(R(e_a,e_b) e_c, e_d)g(R(e_a,e_b)e_c,e_d),
\end{equation}
in any orthonormal basis,
writes, in the fixed adapted basis, as
\begin{eqnarray}
\begin{array}{ll}
\Vert R \Vert^2 &= 2[(\frac{c+3}{4})^2
+(\frac{c-1}{4})^2]\sum_{i,j,k,m}
(\delta_{ik}\delta_{jm}-\delta_{im}\delta_{jk})^2 +
8\sum_{ij}\delta_{ij} \cr
\\
& \quad + 4 \sum_{i,j,k,m} (\frac{c+3}{4}\delta_{ik}\delta_{jm} + \frac{c-1}{4}\delta_{im}\delta_{jk}
+ \frac{c-1}{2}\delta_{ij}\delta_{km})^2\cr
\\
& = [\frac{(c+3)^2}{8}+\frac{(c-1)^2}{8}](2n^2 -2n)+8n +
4\{\frac{(c+3)^2}{16}n^2 + \frac{(c-1)^2}{16}n^2 \cr
\\
& \quad + \frac{(c-1)^2}{4}n^2 + 2\frac{c+3}{4}(\frac{c-1}{4}n +\frac{c-1}{2} n)+ \frac{(c-1)^2}{4}n\}\cr
\\
&= \frac{[(c-1)^2+(c+3)^2]n(n-1)}{4} + \frac{(c+3)^2}{4}n^2 +
\frac{(c-1)^2}{4}(5n^2 + 4n)+ \frac{3(c+3)(c-1)n}{2}+8n \cr
\\
& =\frac{(c+3)^2n(2n-1)}{4} + \frac{(c-1)^2n(6n+3)}{4} +
\frac{3(c+3)(c-1)n}{2} + 8n.
\end{array}
\end{eqnarray}
Furthermore, for the norm of the Ricci tensor, computing
$$\Vert \rho \Vert^2 = \sum_{a,b} \rho(e_a, e_b) \rho (e_a, e_b),$$
in the adapted basis, one has
\begin{equation}
\Vert \rho \Vert^2 = 2 \left(\frac{n(c+3)+c-1}{2}\right)^2 n +4n^2\,.
\end{equation}
Summarizing the above computations, by Theorem \ref{T1} we get the following proposition.
\begin{prop}\label{ab}
If $\mathcal L$ is a Riemannian Legendre foliation on a compact $(2n+1)$-dimensional
Sasakian manifold of constant $\varphi$-sectional curvature $c$, then the spectral invariants satisfy:
\begin{eqnarray}\label{a}
a_0(\Delta_g)& =&  a_0 = Vol_g(M)
\\
a_1(\Delta_g) &=& a_1 = \frac{n((2n+1)(c+3)+c -1)}{12} Vol_g(M)
\\
a_2(\Delta_g) & =& a_2 =
 \frac{n}{1440} \Big(64 - 32 n
  +(c+3)^2(-2 + 9n + 16 n^2 + 20 n^3 )\\
 &&  +(c+3)(c-1)(12 + 2 n + 20 n^2)
   + (c-1)^2(2 + 17 n)\Big) Vol_g(M)
\nonumber
\\
b_0(\mathcal J _{\nabla})& =&b_0 = (n+1)Vol_g(M)
\\
b_1(\mathcal J _{\nabla}) & =&b_1 = (n+1) a_1 + \int_M \tau_{\nabla} dv_g
\\
b_2(\mathcal J _{\nabla}) & = &b_2 = (n+1) a_2 +
 \frac{1}{12}\int_M (2\tau \tau_{\nabla}+6\Vert \rho_{\nabla}\Vert^2 -
 \Vert R_{\nabla}\Vert^2)dv_g\,.\label{b}
\end{eqnarray}
\end{prop}
Since the foliation $\mathcal L$ is assumed to be Riemannian, there
exists, locally, a Riemannian submersion whose vertical and
horizontal distributions are $L$ and $L^{\perp} = \varphi(L) \oplus
[\xi]$, respectively. Let $A$ and $T$ be the O'Neill tensors (see
\cite{FIP}, \cite[p.~49]{ton}).
\begin{prop}\label{minimale}
Let $\mathcal L$ be a Riemannian Legendre minimal foliation on a Sasakian manifold with constant
$\varphi$-sectional curvature $c$. Then
\begin{itemize}
\item[\rm a)] $\tau_{\nabla}=3\Vert A \Vert^2 + \frac{n}{4}((c+3)(n-1)+8)$;
\item[\rm b)] $\Vert A \Vert^2=\Vert T \Vert^2+ n(c+1)$.
\end{itemize}
\end{prop}

\begin{proof} Let $X,Y \in \Gamma (\varphi(L) \oplus [\xi])$. From the theory of Riemannian
submersions \cite{FIP} we have
\begin{equation}
K(X,Y)=K_{\nabla}(X,Y) -3 \Vert A_X Y \Vert^2,
\end{equation}
where $\nabla$ is the connection associated to $TM\,/L \simeq \varphi(L) \oplus [\xi]$,
which in the case of a  Riemannian foliation coincides with the connection induced
by the Levi-Civita connection of $M$ on the horizontal distribution.\\
Fixing an adapted local orthonormal basis of the foliation, by \cite{To}, we see
the transverse scalar curvature can be written as
\begin{equation}\label{taunabla}
\begin{array}{ll}
\tau_{\nabla} & = \sum\limits_{i \neq j =1}^n K(\varphi e_i,\varphi e_j)
+ 2\sum\limits_{i=1}^n K(\varphi e_i, \xi) +3 \Vert A \Vert^2 \cr
\\
& = \sum\limits_{i \neq j =1}^n R(\varphi e _i, \varphi e_j, \varphi e_i, \varphi e_j)
+ 2 \sum\limits_{i=1}^n R(\varphi e_i,\xi,\varphi e_i, \xi) + 3 \Vert A \Vert^2\cr
\\
& = \frac{c+3}{4}n(n-1) +2n +3\Vert A \Vert^2\,,
\end{array}
\end{equation}
which implies a).\\
Denoting by $\tau^{mixed}$ the mixed scalar curvature, defined by
\begin{equation}
\tau^{mixed}= \sum\limits_{i=1}^n \sum\limits_{j=n+1}^{2n}R(e_i,f_j,e_i,f_j)
 + \sum\limits_{i=1}^n R(e_i,\xi,e_i,\xi)\,,
\end{equation}
where $f_{n+i}=\varphi e_i$, we obtain:
\begin{equation}\label{mista}
\tau^{mixed} = \sum_{i,j}c\delta_{ij} +n = (c+1)n\,.
\end{equation}
We denote by $H$ the mean curvature of the leaves. We recall
Ranjan's formula (\cite{Ra})
\begin{equation}\label{ranjan}
\tau^{mixed} = div(H) +\Vert H \Vert^2 + \Vert A \Vert^2- \Vert T
\Vert^2\,.
\end{equation}
Specializing to the case $H=0$, the relations (\ref{ranjan}) and \eqref{mista} simply imply b).
\end{proof}
\begin{theorem}\label{t:2.3} Let $\mathcal L$ and  ${\mathcal L}_0$ be two
Riemannian Legendre minimal foliations on compact Sasakian manifolds
$(M,\varphi,\xi,\eta,g)$ and $(M_0,\varphi_0,\xi_0,\eta_0,g_0)$.
 If $(M,g)$ and $(M,g_0)$ have constant $\varphi$, and $\varphi_0$-sectional curvature $c$ and $c_0$, and
if $\mathcal L$ and ${\mathcal L}_0$ are isospectral, then
\begin{itemize}
\item[\rm a)] $\dim M=\dim M_0$, $Vol(M) = Vol(M_0)$, $c=c_0$,
\item[\rm b)] $\int_M \Vert A \Vert^2 dv_g = \int_{M_0} \Vert A_0 \Vert^2 dv_{g_0}$, and
\item[\rm c)] $\int_M \Vert T \Vert^2 dv_g = \int_{M_0} \Vert T_0 \Vert^2 dv_{g_0}$.
\end{itemize}
\end{theorem}

\begin{proof}
By Theorem \ref{T2}i) and ii), we see that $n=n_0$ and $\mathrm{Vol}(g)=\mathrm{Vol}(g_0)$ and
$$ \int_M \tau\, dv_g=\int_{M_0}\tau_0\, dv_{g_0},
 \text{ and } \int_M \tau_{\nabla} dv_g=\int_{M_0}\tau_{0\, \nabla_0}dv_{g_0}.$$
Therefore by \eqref{e:tau2.3}, we get $c=c_0$,
which by Proposition \ref{minimale}a)  simply implies a). Now, by Proposition \ref{minimale}b) we get b) .
\end{proof}
\begin{cor} Under the hypotheses of Theorem \ref{t:2.3},
the following statements hold:
\begin{itemize}
\item[(a)] If  $\varphi(L) \oplus [\xi] $ is  integrable, then so is $\varphi(L_0) \oplus [\xi]$.\\
\item[(b)] If  $\mathcal L$ is  totally geodesic, then so is  ${\mathcal L}_0$.
\end{itemize}
\end{cor}

\begin{proof} It is sufficient to observe that
$A =0 \Rightarrow A_0 =0$ and that $T =0 \Rightarrow T_0 =0\,.$
\end{proof}
\section{The invariants $b_1$ and $b_2$}
We shall explicitly compute $b_1, b_2$ of Proposition \ref{ab}.
By Theorem \ref{T1} and Proposition \ref{minimale}
 we get
\begin{eqnarray}
\qquad b_1(\mathcal J _{\nabla})  = \frac{n}{12}[(c+3)(2n^2 +6n -2)+3(n+1)(c-1)+2n] + 3\int_M \Vert A \Vert^2 dv_g\, .
\end{eqnarray}

Now, we proceed to the computation of each term involved in
\eqref{b} of $b_2$. Let $(X_i, U_j)$ be an orthonormal basis adapted
to the foliation with $X_i$ horizontal and $U_j$ vertical. We
introduce the notations from \cite{Besse}
\begin{eqnarray}
(A_X,A_Y)&=& \sum_i g(A_X X_i,A_Y X_i)=\sum_j g(A_X U_j,A_Y U_j)\,,\\
 (TX,TY) &=& \sum_j g(T_{U_j}X,T_{U_J}Y).
\end{eqnarray}
Since $\mathcal L$ is assumed to be minimal,  by Proposition 9.36 in \cite{Besse}, we have
\begin{equation}
\rho_{\nabla}(X,Y)=\rho(X,Y)+ 2 (A_X,A_Y)+(TX,TY),
\end{equation}

Setting $f_i=\varphi e_i$, for $i \in \{1,\ldots,n\}$, and $f_{n+1}=\xi$, we have
\begin{equation}
\Vert \rho_{\nabla}\Vert^2 = \sum\limits_{i,j=1}^n \rho_{\nabla} (f_i,f_j)^2 + \rho_{\nabla}(\xi,\xi)^2
\end{equation}
and
we obtain
\begin{eqnarray}
\begin{array}{ll}
\Vert \rho_{\nabla} \Vert^2 & = \sum\limits_{i,j=1}^{n+1} \rho
(f_i,f_j)^2 + 2 \sum\limits_{i,j=1}^{n+1}
\rho(f_i,f_j)[2(A_{f_i},A_{f_j})+(T f_i,T f_j)]\cr
\\
& \quad + \sum\limits_{i,j =1}^{n+1} [2(A_{f_i},A_{f_j})+(Tf_i,Tf_j)]^2\,.
\end{array}
\end{eqnarray}
We easily see that
$$T_{e_i}\xi =0,\quad T_{e_i}\varphi e_j =\varphi(T_{e_i}e_j),
 \quad A_{\varphi e_i} \varphi e_j = \varphi(A_{\varphi e_i} e_j), \quad A_{\xi} \varphi e_i = e_i,$$
$$2(A_{\xi}, A_{\xi})+ (T \xi, T \xi)=2 \sum\limits_{i=1}^n (A_{\xi} \varphi e_i , A_{\xi}\varphi e_i)=2n.$$
We also know from (\ref{ric}) that
$$\rho(\xi,\xi)= 2n$$
$$\sum\limits_{i,j=1}^n \rho(f_i,f_j)^2 =\left[\frac{n(c+3)+c-1}{2}\right]^2 n.$$
Setting
$$l=\sum\limits_{i,j=1}^{n+1} \rho(f_i,f_j)^2,$$
we obtain:
\begin{eqnarray}\label{nablarho}
\begin{array}{ll}
\Vert \rho_{\nabla} \Vert^2 & =
l +2\sum_{i,j}\frac{n(c+3)+c-1}{2}\delta_{ij}\left[2(A_{f_i},A_{f_j})+(T f_i,T f_j)\right]\cr
\\
& \quad +2(2n)^2 + \sum\limits_{i,j=1}^{n+1}\left[2(A_{f_i},A_{f_j})+(Tf_i,T f_j)\right]^2\cr
\\
& =(l+8n^2) + (n(c+3)+c-1)\left[2 \Vert A \Vert^2 + \Vert T \Vert^2\right]\cr
\\
& \quad + \sum\limits_{i,j=1}^{n+1}\left[2(A_{f_i},A_{f_j})+(T f_i,T f_j)\right]^2\,.
\end{array}
\end{eqnarray}

\begin{prop} If  $\mathcal L$ is a  Riemannian foliation on $(M,g)$, then
  \begin{equation}\label{curvatura}
  \sum\limits_{i=1}^n R(X,e_i,Y,e_i)  = \frac{1}{2}(g(\nabla^M_Y H,X)+ g(\nabla^M_X H,Y))
   + (A_X,A_Y)- (TX,TY),
  \end{equation}
 where $\nabla^M$ is the Levi-Civita connection of $(M,g)$, $H$ is the mean curvature of the leaves,
 $(e_1,\ldots,e_n)$ is a local basis of the vertical distribution $L$,
  and  $X,Y$ are horizontal.
\end{prop}

\begin{proof} From the theory of Riemannian submersions, for  any horizontal
vectors $X,Y$ and vertical vector $U$, we have
\begin{eqnarray*}
R(X,U,Y,U) & =&g((\nabla^M_X T)_U U,Y)-g(T_U X,T_U Y)
+ g((\nabla^M_U A)_X Y, U)+g(A_X U, A_Y U).
\end{eqnarray*}
Therefore
\begin{eqnarray*}
\begin{array}{ll}
\sum\limits_{i=1}^n R(X,e_i,Y,e_i)& =\sum\limits_{i=1}^n
(g((\nabla^M_X T)_{e_i}e_i,Y)-g(T_{e_i} X,T_{e_i} Y)\cr
\\
& \quad + g((\nabla^M_{e_i} A)_X Y, e_i)+g(A_X e_i, A_Y e_i)).
\end{array}
\end{eqnarray*}
The covariant derivative of $T$ satisfies
\begin{eqnarray*}
\sum_i g((\nabla^M_X T)_{e_i} e_i,Y) &=& g(\nabla^M_X H,Y) - \sum_i
g(T_{v(\nabla^M_X e_i)} e_i,Y)- \sum_i g(T_{e_i} v (\nabla^M_X
e_i),Y)
\\
& = & g(\nabla^M_X H,Y)-2\sum_i g(T_{e_i} v (\nabla^M_X e_i) , Y)\,,
\end{eqnarray*}
since $T_U W =T_W U$ and $ \sum_i T_{e_i} e_i = H$.\\
Setting $$v(\nabla^M_X e_i) =\sum_j h_{ij} e_j$$ we note that
$$h_{ij}=g(v(\nabla^M_X e_i),e_j)= X(g(e_i,e_j))- g(e_i, \nabla^M_X e_j)=-g(e_i, v(\nabla^M_X e_j))= -h_{ji};$$
$$\sum_{i,j} g(T_{e_i}h_{ij} e_j,Y)= \sum_{i,j}h_{ij}g(T_{e_i}e_j,Y)=0;$$
$$\sum_i g((\nabla^M_{e_i} A)_X Y, e_i)= \frac{1}{2}\left(g(\nabla^M_Y H,X)-g(\nabla^M_X H, Y)\right).$$
Therefore
\begin{eqnarray}\label{curvat}
\begin{array}{ll}
\sum_i R(X,e_i,Y,e_i) & = g(\nabla^M_X H,Y)-(TX,TY)+ (A_X,A_Y)\cr
\\
& \quad + \frac{1}{2}(g(\nabla^M_Y H,X) -g(\nabla^M_X H,Y)),
\end{array}
\end{eqnarray}
and \eqref{curvatura} follows.
\end{proof}

From \eqref{curvat}, one can obtain the following proposition.
\begin{prop}
If $\mathcal L$ is a Riemannian foliation with minimal leaves then
\begin{equation}\label{minimale1}
\sum\limits_{k=1}^n R(f_i,e_k,f_j,e_k)= (A_{f_i},A_{f_j})- (T f_i,T
f_j).
\end{equation}
\end{prop}

Now, we consider a Riemannian Legendre foliation with minimal leaves
on a Sasakian manifold $M$ of constant $\varphi$-sectional curvature
$c$ and we fix a local orthonormal basis $(e_i,f_i,f_{n+1})$ adapted
to the foliation, that is $f_{n+1}=\xi$ and $f_i=\varphi e_i$ for
any $i\in\{1,\ldots,n\}$ and $\{e_1 ,\ldots , e_n\}$ is a local
basis of a leaf $L$.

Setting
$$ S(f_i,f_j) = \sum\limits_{k=1}^n R(f_i,e_k,f_j,e_k)$$
and using relations (\ref{curvature}), we obtain that
$$S(\varphi e_i,\xi)= \sum\limits_{k=1}^n R(\varphi e_i,e_k,\xi,e_k)=0,$$
$$ S(\xi,\xi)=n,\quad S(f_i,f_j)=S(\varphi e_i,\varphi e_j)= d \delta_{ij},$$
where $d= \frac{c+3}{4}n + \frac{3(c-1)}{4}$.\\

By Proposition  \ref{minimale}b) and equations \eqref{nablarho} and
\eqref{minimale1}, it follows:
\begin{eqnarray}
\begin{array}{ll}
\Vert \rho_{\nabla}\Vert^2 & = (l+8n^2)+ [n(c+3)+c-1][3\Vert A \Vert^2 -c(n+1)] \cr
\\
& \quad + \sum\limits_{i,j=1}^{n+1}[3(A_{f_i},A_{f_j})-S(f_i,f_j)]^2.
\end{array}
\end{eqnarray}
Denoting by $E$ the last sum of the previous relation, we have:
\begin{eqnarray}
\begin{array}{ll}
\qquad\qquad E & = \sum\limits_{i=1}^n [3(A_{\varphi e_i},A_{\varphi
e_i})-d]^2 +[2(A_{\xi},A_{\xi})+ (T \xi ,T \xi)]^2
+2\sum\limits_{i<j}^{n+1}[3(A_{f_i},A_{f_j})]^2 \cr
\\
& = 9\sum\limits_{i=1}^n(A_{\varphi e_i},A_{\varphi e_i})^2 + d^2n
-6d\sum\limits_{i=1}^n (A_{\varphi e_i},A_{\varphi e_i}) + 4n^2
+ 18\sum\limits_{i<j}^{n+1} (A_{f_i},A_{f_j})^2.
\end{array}
\end{eqnarray}

Summarizing, we conclude:
\begin{prop} Under the hypothesis of Proposition \ref{ab},
the following holds
\begin{eqnarray}
\Vert\rho_{\nabla}\Vert^2
 &=&  9\sum\limits_{i=1}^n(A_{\varphi
      e_i},A_{\varphi e_i})^2 + 18\sum\limits_{i<j}^{n+1}
      (A_{f_i},A_{f_j})^2
 \\
 && + nd(d+6) -6d\Vert A \Vert^2+ 16n^2 \nonumber
 \\
 &&+\left(n(c+3)+c-1\right)\left(3\Vert A \Vert^2 -c(n+1)\right)\nonumber
 \\
 &&+ n\left(\frac{n(c+3)+c-1}{2}\right)^2, \nonumber
\end{eqnarray}
where $d= \frac{c+3}{4}n + \frac{3(c-1)}{4}$.
\end{prop}

To compute $\Vert R_{\nabla}\Vert^2$, we notice that
\begin{eqnarray*}
\begin{array}{ll}
R_{\nabla}(f_i,f_j,f_k,f_l) & =R(f_i,f_j,f_k,f_l)+2g(A_{f_i}f_j,A_{f_k}f_l)\cr
\\
& \quad - g(A_{f_j}f_k,A_{f_i}f_l)-g(A_{f_k}f_i,A_{f_j}f_l)\,,
\end{array}
\end{eqnarray*}
we consider the following tensor of type $(0,4)$ associated to the horizontal distribution
$$V(X,Y,Z,Z')= 2g(A_X Y,A_Z Z')- g(A_Y Z,A_X Z')-g(A_Z X ,A_Y Z')$$
and we set $$\Vert V \Vert^2 =
\sum_{i,j,k,l=1}^{n+1}V(f_i,f_j,f_k,f_l)^2.$$ We can write:
\begin{eqnarray*}
\begin{array}{ll}
\Vert R_{\nabla}\Vert^2 & = \sum\limits_{i,j,k,l=1}^{n+1}(R_{\nabla}(f_i,f_j,f_k,f_l))^2 \cr
\\
& =\sum\limits_{i,j,k,l=1}^{n+1}(R(f_i,f_j,f_k,f_l))^2
+ 2\sum\limits_{i,j,k,l=1}^{n+1}R(f_i,f_j,f_k,f_l)\{2g(A_{f_i}f_j,A_{f_k}f_l)\cr
\\
& \quad - g(A_{f_j}f_k,A_{f_i}f_l)-g(A_{f_k}f_i,A_{f_j}f_l)\} +\Vert V \Vert^2\,.
\end{array}
\end{eqnarray*}
Then
\begin{eqnarray}
l'&=& \sum\limits_{i,j,k,l=1}^{n+1} R(f_i,f_j,f_k,f_l)^2\\
&=&\sum\limits_{i,j,k,l=1}^{n} R(\varphi e_i,\varphi e_j,\varphi
e_k,\varphi e_l)^2  + 4 \sum\limits_{i,k=1}^n R(\varphi e_i, \xi,
\varphi e_k, \xi)^2 \nonumber
\\
& =& (\frac{c+3}{4})^2
\sum\limits_{i,j,k,l=1}^n(\delta_{jk}\delta_{il}-\delta_{ik}\delta_{jl})^2
+ 4\sum_{i,k}\delta_{ik}^2\nonumber
\cr\\
& =& \frac{(c+3)^2}{16}[2n^2 -2 \sum_{i,j}\delta_{ij}]+4n=
\frac{(c+3)^2(n-1)n}{8} + 4n\,,\nonumber
\end{eqnarray}
and thus,
\begin{eqnarray*}
\Vert R_{\nabla}\Vert^2 & = &l'
+2\sum\limits_{i,j,k,l=1}^{n}\frac{c+3}{4}(\delta_{ik}\delta_{jl}-\delta_{il}\delta_{jk})[2g(A_{\varphi
e_i}\varphi e_j,A_{\varphi e_k}\varphi e_l)
\\
&&  -g(A_{\varphi e_j}\varphi e_k,A_{\varphi e_i}\varphi
e_l)-g(A_{\varphi e_k}\varphi e_i,A_{\varphi e_j}\varphi e_l)]
\\
&&   + 8 \sum_{i,k} R(\varphi e_i, \xi, \varphi e_k,
\xi)[2g(A_{\varphi e_i} \xi,A_{\varphi e_k} \xi) - g(A_{\xi}\varphi
e_k,A_{\xi}\varphi e_i)]
\\
& = &l' + \Vert V \Vert^2 + 2\frac{c+3}{4}\sum_{i,j}[2g(A_{\varphi
e_i}\varphi e_j,A_{\varphi e_i}\varphi e_j)
\\
&&   - g(A_{\varphi e_j}\varphi e_i,A_{\varphi e_i}\varphi
e_j)-g(A_{\varphi e_i}\varphi e_i,A_{\varphi e_j}\varphi e_j)]
\\
&&  -2\frac{c+3}{4}\sum_{i,j}[2g(A_{\varphi e_i}\varphi
e_j,A_{\varphi e_j}\varphi e_i) \cr
\\
&&  - g(A_{\varphi e_j}\varphi e_j,A_{\varphi e_i}\varphi
e_i)-g(A_{\varphi e_j}\varphi e_i,A_{\varphi e_j}\varphi e_i)] + 24n
\\
& =&l' + 24n -3n(c+3) +3(c+3)\Vert A \Vert^2 + \Vert V \Vert^2\,.
\end{eqnarray*}
Therefore
\begin{equation}\label{Rnabla}
\Vert R_{\nabla}\Vert^2= \frac{(c+3)^2(n-1)n}{8}  -3n(c+3)
+28n+3(c+3)\Vert A \Vert^2 + \Vert V \Vert^2\,.
\end{equation}
On the other hand, we get
\begin{eqnarray*}
\begin{array}{ll}
\sum_j V(f_i,f_j,f_k,f_j)&=  \sum_j [2g(A_{f_i}f_j,A_{f_k}f_j)
 - g(A_{f_j}f_k,A_{f_i}f_j)-g(A_{f_k}f_i,A_{f_j}f_j)]\cr
\\
& = 3\sum_j g(A_{f_i}f_j,A_{f_k}f_j) =3(A_{f_i},A_{f_k})\,,
\end{array}
\end{eqnarray*}
and thus
\begin{equation}\label{e:last}
(C_{24} V)(f_i,f_k) =3(A_{f_i},A_{f_k})\,,
\end{equation}
where $C_{24}$ denotes the contraction of the tensor with respect to
the indices 2 and 4. The Hilbert-Schmidt norm of the $(0,2)$ tensor
$C_{24} V$ along the horizontal distribution satisfies
\begin{eqnarray}\label{e:2.36}
\Vert C_{24} V \Vert^2&=& 9\sum\limits_{i=1}^n(A_{\varphi
e_i},A_{\varphi e_i})^2 + 18\sum\limits_{i<j}^{n+1}
(A_{f_i},A_{f_j})^2 +9n^2.
\end{eqnarray}
Summarizing, by Proposition \ref{ab}, we obtain that
\begin{eqnarray}\label{e:new}
b_2&=& (n+1)a_2+\frac{1}{12}\int_M 2\tau \tau_{\nabla}+6\Vert
\rho_{\nabla}\Vert^2 - \Vert
R_{\nabla}\Vert^2 dv_g\\
&=&(n+1)a_2+\frac{1}{12}\int_M 6\Vert  (C_{24} V)\Vert ^2- \Vert V \Vert^2 dv_g\nonumber\\
&&
+\frac{((c+3)(-3 + 12  n + 6  n^2 ) + (c-1)(-9 + 3 n ))}{12}\int_M \Vert A \Vert^2 dv_g\nonumber\\
&&+\frac{\mathrm{Vol_g(M)}}{12}\Big(42 n^2- 28 n+
  (c+3)n(3  + 11 n + 4 n^2)\nonumber\\
&& + \frac{1}{8} (c+3)^2n(-11  - 15 n  + 13n^2  + 4n^3 )\nonumber\\
  && +
  (c-1)n(27  + 2 n) + \frac{1}{8} (c-1)^2(-36 + 3 n  )\nonumber \\&&
   + \frac{1}{4} (c+3)(c-1)(-6 - 24 n + 2 n^2 +
        n^3) \Big).\nonumber
\end{eqnarray}

By \eqref{e:new}  and Theorems \ref{T1}, \ref{T2}, \ref{t:2.3}, we
now get our main result.
\begin{theorem}\label{t:main}
 Let $(M^{2n+1}, \varphi, \xi,\eta,g)$ and $(M_0^{2n_0+1},\varphi_0, \xi_0, \eta_0, g_0)$ be
compact isospectral Sasakian manifolds with constant
$\varphi$-sectional curvature $c$ and  constant $\varphi_0$-sectional curvature
$c_0$ respectively. If $\mathcal L$ and ${\mathcal L}_0$ are
 Riemannian minimal Legendre  foliations  on $M$ and $M_0$
 such that $Spec({\mathcal L},{\mathcal J}_{\nabla})=Spec({\mathcal L}_0,{\mathcal J}_{\nabla_0})$, then
\begin{itemize}
\item[\rm 1)] $\dim\, M = \dim\, M_0\,, \quad Vol(M)=Vol(M_0)\,,\quad c=c_0$,
\item[\rm 2)] $\int_M \Vert A \Vert^2 dv_g
= \int_{M_0} \Vert A_0 \Vert^2 dv_{g_0}\,,\quad \int_M \Vert T \Vert^2 dv_g
= \int_{M_0} \Vert T_0 \Vert^2 dv_{g_0}$,
\item[\rm 3)] $\int_M [6\Vert C_{24} V\Vert^2 - \Vert V \Vert^2]dv_g
= \int_{M_0} [6 \Vert C_{24} V_0\Vert^2 - \Vert V_0 \Vert^2]dv_{g_0}$.
\end{itemize}
\end{theorem}
\section{Concluding remarks}

Let $\mathcal L$ be a Riemannian Legendre foliation with totally
geodesic leaves on $(M,g)$ and assume that $M$ has the constant
curvature $c=1$ and that $2n+1=\dim M$. In this particular case, we
would like to point that the condition 3) of Theorem \ref{t:main} is
implied by 1). Indeed, for any Riemannian totally geodesic foliation
$\mathcal L$ on a constant curvature space $M$ with
 $\dim Q=\dim L+1$, one can see that
 $$g(A_YW,A_YW)=g(Y,Y)g(W,W),\ \mathrm{for\ any}\ W\in L\ \mathrm{ and\ for\ any}\ Y\in Q,$$
 $A_X:L\to Q^{\perp X}=\{Y\in Q \,|\, g(Y,X)=0\}$ is a bijection for any
 unit vector $X$,  and $R_\nabla$ has the constant curvature $4$ (see the
 argument of \cite[Prop.~4.6]{baditoiu}). Thus, we simply have
\begin{eqnarray*}
 (A_{\varphi e_i},A_{\varphi e_i})&=&\sum_{k=1}^{n} g(\varphi e_i,\varphi
 e_i)g(e_k,e_k)=n,\\
(A_{ f_i},A_{f_j})&=&\sum_{k=1}^{n+1} g(f_i,f_j)g(e_k,e_k)=0,\ \
\mathrm{for\ any }\ i<j,
\end{eqnarray*}
and therefore, by \eqref{e:2.36}, $\Vert  C_{24}V\Vert ^2=18n^2$. We
easily see that
\begin{eqnarray*}
\Vert V\Vert ^2&=&\sum_{i,j,k,l=1}^{n+1}V(f_i,f_j,f_k,f_l)^2=\sum_{i,j,k,l=1}^{n+1}(R_\nabla(f_i,f_j,f_k,f_l)-R(f_i,f_j,f_k,f_l))^2\\
    &=&\sum_{i,j,k,l=1}^{n+1}[(4-1)(\delta_{ik}\delta_{jl}-\delta_{il}\delta_{jk})]^2=18n^2+18n.
\end{eqnarray*}
This concludes that
\begin{eqnarray}
 \int_M [6\Vert C_{24} V\Vert^2 - \Vert V
\Vert^2]dv_g=(90n^2-18n)Vol(M).
\end{eqnarray}

\begin{example}\label{example1}
Let $(S^3,\varphi, \xi,\eta,g)$ be the standard contact metric
structure on $S^3$, which we now recall. Let $\mathbb
H=\{x_1+ix_2+jx_3+kx_4\,|\, x_1,x_2,x_3,x_4\in\mathbb R\}$ be the
algebra of quaternion numbers, where $i^2=j^2=k^2=-1,\ ij=-ji=k$.
The set of unit quaternions is identified with $S^3$. Let $(I,J,K)$
be the quaternionic structure on $\mathbb H$ given for
 $I,J,K:\mathbb H\to\mathbb H$, by $I(h)=ih$, $J(h)=jh$, $K(h)=kh$ for any $h\in
\mathbb H$. Let $N$ be the unit outer normal vector field on $S^3$ and let
$g$ be the Riemann metric with constant curvature $c=1$. We set
$\xi=-IN$, $\eta$ the dual form of $\xi$; $\varphi(Z)$ the
projection of $I(Z)$ onto tangent space of $S^3$, for any vector
field $Z$ of $S^3$. Note that $(\varphi, \xi,\eta,g)$ is the
standard contact metric on $S^3$ and the its $\varphi$-sectional
curvature is $c=1$.

Let $(x_1,x_2,x_3,x_4)$ be the Cartesian coordinate system on
$\mathbb R^4=\mathbb H$. It is easy to see that
 $$\xi=x_2\frac{\partial}{\partial x_1}-x_1\frac{\partial}{\partial x_2}
+x_4\frac{\partial}{\partial x_3}-x_3\frac{\partial}{\partial x_4}.
 $$
Setting $W=-JN$ and $Y=-KN$, we have
$$W=x_3\frac{\partial}{\partial x_1}-x_4\frac{\partial}{\partial x_2}
 -x_1\frac{\partial}{\partial x_3}
 +x_2\frac{\partial}{\partial x_4};$$
 $$ Y=x_4\frac{\partial}{\partial
x_1}+x_3\frac{\partial}{\partial x_2} -x_2\frac{\partial}{\partial
x_3}-x_1\frac{\partial}{\partial x_4}.
 $$
One can easily compute the Lie brackets between these  vectors:
$[W,\xi]=-2Y$, $[Y,\xi]=2W$, $[W,Y]=2\xi$ (see \cite{jayne}). The
distributions $L=\mathrm{span}\{W\}$ and $L'=\mathrm{span}\{Y\}$
define two non-degenerate Legendre foliations (see \cite[Example
7.1]{jayne}). Since $[W,Y]=2\xi$ and $\varphi(W)=Y$, by \cite[Lemma
6.6]{jayne}, $L$ and $L'$ are Riemannian Legendre foliations on the
Sasakian space form $S^3$ and both of them are totally geodesic. By
theorem \ref{t:main}, any  Riemannian minimal Legendre  foliations
on a compact Sasakian space form $M_0$ with $\varphi_0$-sectional
curvature $c_0$, isospectral to the foliation $\mathcal L$ (defined
above) on the standard Sasakian space form $(S^3,\varphi,\xi,\eta)$
is totally geodesic, $c_0=1$, $\dim M_0=3$, and
 $$\int_M [6\Vert C_{24} V\Vert^2
- \Vert V \Vert^2]dv_g = 72 \mathrm{Vol}(M_0)=72 \mathrm{Vol}(M),$$
$$\int_M \Vert A\Vert^2 dv_g =2Vol(M).$$
\end{example}

It is well known that a typical example of a Sasakian space form is a $\mathcal
D$-homothetic deformation of the standard contact metric structure
of an odd-dimensional sphere $S^{2n+1}$, which we now recall (see
\cite[Example 7.4.1]{Bl1}). For a contact metric structure
$(\varphi,\xi,\eta,g)$, one defines the $\mathcal D$-homothetic
deformation $(\bar\varphi,\bar\xi,\bar\eta,\bar g)$
 by
$$
 \bar\varphi=\varphi,\ \ \
 \bar\xi=\frac{1}{a}\xi,\ \ \
 \bar\eta=a\eta,\ \ \
 \bar g=ag+a(a-1)\eta\otimes\eta,
$$
where $a$ is a positive constant (see \cite[p.~114]{Bl1}). By
\cite[Theorem 7.15]{Bl1}, a compact simply connected Sasakian space
form with $\varphi$-sectional curvature $c>-3$ is a $\mathcal
D$-homothetic deformation of the standard contact metric structure
on $S^{2n+1}$ and $c=\frac{4}{a}-3$ (for some $a>0$). Since
$\mathrm{Ker}\, \bar\eta=\mathrm{Ker}\, \eta$, the problem of
finding Riemannian Legendre foliations on such a compact Sasakian
space form (with $c>-3$) reduces to the one on the standard sphere
$S^{2n+1}$ (i.e. c=1).

{\bf Case $n=1$}. One can apply a $\mathcal D$-homothetic
deformation to Example 1 to obtain an example   for any $c>-3$.

{\bf Case $n=2$}. From \cite{gg}, there are no Riemannian foliations
with two-dimensional leaves on a standard sphere, which in
particular means that there are no Riemannian Legendre foliations on
  $S^5$.

{\bf Case $n=3$}. By \cite[Theorem 5.3]{gg}, we get, in particular,
  that any Riemannian foliation with 3-dimensional leaves
on $S^7$  is given uniquely (up to equivalence) by a direct  sum of
irreducible unitary representations of $SU(2)$, namely
$\rho_1\oplus\rho_1$, or by $\rho_3$, where $\rho_k$ is the action
of $SU(2)$ on the set of complex homogeneous polynomials  in two
variables  of degree $k$.\\
 Note that $\rho_1\oplus\rho_1$
corresponds to the Hopf fibration $S^7\to S^4$ (see \cite{gg}),
which is not a Legendre foliation. In fact, no leaf of
$\rho_1\oplus\rho_1$ is Legendrian, simply because $L$, the tangent
distribution of the leaves, is generated by $-\xi=IN$, $JN$, $KN$,
where $(I,J,K)$ is the standard quaternionic structure on $\mathbb
H^2=\mathbb R^8$ and $N$ is the unit outer vector
field to $S^7$ (see \cite[p.~265]{esc}).\\
In \cite[p.~365]{ohn}, Ohnita constructed a unique minimal
Legendrian orbit on $S^7$ under the action of $\rho_3$, which  means
that  only one leaf of the Riemannian foliation given by $\rho_3$ on
$S^7$ is both minimal and Legendrian. This concludes that $\rho_3$
does not provide a Riemannian Legendre foliation with minimal
leaves.\\

Finally, another typical example of a Riemannian Legendre foliation with
totally geodesic leaves is given by the tangent sphere bundle
$\pi:T_1P\to P$ of a Riemannian manifold $(P,h)$. Assume that $T_1P$
is endowed with the standard contact metric structure.

If $\dim P>2$, then $T_1P$ is never a Sasakian space form (see
\cite{opr}). Note that if $P$ has constant curvature, then $T_1P$
admits a non-Sasakian contact metric structure of constant
$\varphi$-sectional curvature $c^2$ if and only if $c=2\pm\sqrt{5}$
(see \cite[Theorem 9.9]{Bl1}).

If $\dim P=2$ and if $T_1P$ is a Sasakian space form, then $P$ has
constant curvature $c=1$ (see \cite[Theorem 9.3]{Bl1}). Note that
$T_1S^2\simeq\mathbb RP^3$ (see \cite[p.~142]{Bl1}) and the
Riemannian Legendre foliation on the universal cover of $T_1S^2$ is
equivalent to Example 1.

\begin{acknowledgments} Gabriel B\u adi\c toiu would like to thank Professor
Stefano Marchiafava for hospitality, support and useful discussions
on this topic. The first author was supported by grant COFIN PRIN
2007 ``Geometria Riemanniana e Strutture Differenziabili'' and by a
grant of the Romanian National Authority for Scientific Research,
CNCS - UEFISCDI, project number PN-II-ID-PCE-2011-3-0362. The second
author was supported by grant CNCSIS 34699/2006 in Romania.
\end{acknowledgments}

\end{document}